\documentclass[12pt]{article}
\usepackage{dsfont}
\usepackage{amsfonts}
\usepackage{amsthm}
\usepackage{amssymb}
\usepackage{mathtools}
\usepackage{amsmath}
\usepackage{mathrsfs}
\usepackage{mathrsfs}
\usepackage{stmaryrd}
\usepackage{framed}
\usepackage{nccmath}
\usepackage{wrapfig}
\usepackage{enumitem}
\usepackage{geometry}
\usepackage{tikz,tikz-cd}
\usepackage{bm}
\usepackage[stable]{footmisc}
\usepackage[new]{old-arrows}
\usepackage{draftwatermark}
\usepackage{quiver}
\SetWatermarkText{DRAFT - GSN - 2022}
\SetWatermarkColor[gray]{0.95}
\SetWatermarkScale{0.4}
\NewDocumentCommand{\tens}{t_}
 {%
  \IfBooleanTF{#1}
   {\tensop}
   {\otimes}%
 }
\NewDocumentCommand{\tensop}{m}
 {%
  \mathbin{\mathop{\otimes}\displaylimits_{#1}}%
 }

\DeclareMathOperator{\Spec}{Spec}

\newcommand{\prtt}[1]{\left( #1 \right)}
\newcommand{\tA}[2]{#1 \tens_{R} #2}

\newcommand{\lips}[2]{\tA{#1}{#2}-\tA{#2}{#1}}

\newcommand{\bA}{{\mathbb{A}}}
\newcommand{\bC}{{\mathbb{C}}}

\newcommand{\bN}{{\mathbb{N}}}
\newcommand{\bZ}{{\mathbb{Z}}}
\newcommand{\sub}{\subseteq}
\newcommand{\ten}{\otimes}
\newcommand{\fp}{\mathfrak{p}}
\newcommand{\fq}{\mathfrak{q}}
\newcommand{\bx}{\mathbf{x}}

\newcommand{\id}{\mbox{id}}

\newcommand{\spec}{\mbox{Spec}}
\newcommand{\Frac}{\mbox{Frac}}

\theoremstyle{plain}

\newtheorem{defi}{Definition}[section]
\numberwithin{defi}{section}

\newtheorem{prop}[defi]{Proposition}
\newtheorem{teo}[defi]{Theorem}
\newtheorem{cor}[defi]{Corollary}
\newtheorem{lema}[defi]{Lemma}
\newtheorem{example}[defi]{Example}

\begin{document}
	
		\title{A survey on relative Lipschitz saturation of algebras and its relation with radicial algebras}
	\author{Thiago da Silva and Guilherme Schultz Netto}
	\date{}
	
	\maketitle
	
	\begin{abstract}
		In this work, we introduce the concept of relative Lipschitz saturation, along with its key categorical and algebraic properties, and demonstrate how such a structure always gives rise to a radicial algebra.	
	\end{abstract}

	\section*{Introduction}
	
	The study in equisingularity was started by Zariski in \cite{zariski1965}, where he was interested in investigating this concept in an algebraic variety along an irreducible (singular) subvariety. First, he dealt with different ways to define equivalent singularities of plane algebroid curves, and after in \cite{zariski1965-2} Zariski worked with algebroid hypersurfaces with a singular point at its origin.
	
	In the meantime, in \cite{zariski1968} introduced an operation on a ring $A$, which he called \textit{saturation}, which consists of passing from $A$ to some ring $\tilde{A}$ lying over $A$ and under the integral closure of $A$ in its total ring of fractions. In that work Zariski showed how useful this operation is for the theory of singularities by means of geometric applications to plane algebroid curves, and more generally, to algebroid hypersurfaces.
	
	Having its applications in his mind, firstly Zariski restricted himself to the case in which $A$ is a local domain of dimension one, and always with the assumption that the base field is algebraically closed and of characteristic zero. Then, in \cite{zariski1971(0), zariski1971, zariski1975} he presented his general theory of saturation, extending several results and showing how to apply them to a more general setup in the theory of singularities.
	
	The core of saturation theory developed by Zariski is to look for a special intermediate algebra between a ring $A$ and its integral closure $\overline{A}$. In \cite{phamt}, in the case of complex analytic algebras, Pham and Teissier observed that the germs of Lipschitz meromorphic functions lie between $A$ and $\overline{A}$. Thus, they studied this algebra from a formal and geometric viewpoint, showing that it coincides with the Zariski saturation in the hypersurface case. So, for a reduced complex analytic algebra $A$ with normalization $\overline{A}$, Pham and Teissier defined the Lipschitz saturation of $A$ as $$A^*:=\{f\in\overline{A}\mid f\ten_{\bC}1-1\ten_{\bC}f\in\overline{I_A}\},$$
	
	\noindent where $I_A$ denotes the kernel of the canonical map $\overline{A}\ten_{\bC}\overline{A}\rightarrow \overline{A}\ten_A\overline{A}$. Although Pham and Teissier were thinking with an analytic background, they left a good question about if $A^*_B\sub \overline{A}$, in the case where $B$ is an $A$-algebra and $$A^*_B:=\{f\in B\mid f\ten_{\bC}1-1\ten_{\bC}f\in\overline{I_{A, B}}\},$$
	
	\noindent where $I_{A,B}$ now denotes the kernel of the canonical map $B\ten_{\bC}B\rightarrow B\ten_AB$. 
	
	So, in \cite{lipman1} Lipman extended this definition for a sequence of ring morphisms $R\rightarrow A\overset{g}{\rightarrow}B$, and defined what he called the \textit{relative Lipschitz saturation of $A$ in $B$}, denoted by $A^*_{B, R}$. Besides,  Lipman developed several techniques and general properties on this operation in the ring $A$.
	
	Recently, Gaffney used this machinery to deal with bi-Lipschitz equisingularity of families of curves in \cite{gaffney1}, defining a notion of Lipschitz saturation for an ideal of a complex analytic algebra. After that, Gaffney showed this Lipschitz saturation is related to the integral closure of the \textit{double} of the ideal, a concept that he defined and used to get a type of infinitesimal Lipschitz condition for a family of complex analytic hypersurfaces in \cite{gaffney2}.
	
	The aim of this work is to introduce the concept of relative Lipschitz saturation and demonstrate that such a construction always results in a radicial algebra. Another objective of this work is to enhance the understanding of the presented results by providing additional details and auxiliary lemmas to support the central arguments of the main proofs.
	
	In the first section, we provide the basic definitions and the most immediate results, along with the presentation of some classic examples.
	
	In the second section, we present the main results regarding the categorical and algebraic properties of Lipschitz saturation. Finally, in the third section, we demonstrate how Lipschitz saturation serves as a source of radicial algebras.

	\section{Basic properties} 
	
	Let $R$ be a ring and let $A,B$ be $R$-modules and consider a sequence of ring morphisms 
	\begin{align*}
		R \overset{\tau}{\longrightarrow} A \overset{g}{\longrightarrow} B.
	\end{align*} Consider the map from $B$ to its tensor product by $R$, in a diagonal way:
	\begin{align*}
		\Delta : B &\longrightarrow \tA{B}{B} \\
		b &\longmapsto \lips{b}{1}.
	\end{align*}
	
	It is easy to conclude the following properties of $\Delta$:
	
	\begin{enumerate}
		\item $\Delta(b_1+b_2)=\Delta(b_1+b_2), \forall b_1,b_2\in B$;
		
		\item $\Delta(rb)=r\Delta(b), \forall r\in R$ and $b\in B$;
		
		\item (Leibniz rule) $\Delta(b_1b_2)=(b_1\ten_R 1)\Delta(b_2)+(1\ten_Rb_2)\Delta(b_1), \forall b_1,b_2\in B$. 
	\end{enumerate}
	
	By the universal property of the tensor product, there exists a unique $R$-algebras morphism $$\varphi: \tA{B}{B} \longrightarrow B\tens_{A}B$$
	
	\noindent which maps $x\otimes_R y\mapsto x\otimes_A y$, for all $x,y\in B$.
	
	Furthermore, we already know (see \cite{kleiman}, 8.7) that $$\ker\varphi=\langle ax\ten_R y-x\ten_Ray\mid a\in A\mbox{ and }x,y\in B \rangle.$$
	
	Notice that $ax\ten_R y-x\ten_Ray=(x\ten_Ry)(g(a)\ten_R1-1\ten_Rg(a))$. Therefore,$$\ker\varphi=\langle g(a)\ten_R1-1\ten_Rg(a)\mid a\in A\rangle=\Delta(g(A))(B\ten_RB),$$
	
	\noindent i.e, $\ker\varphi$ is the ideal of $B\ten_RB$ generated by the image of $\Delta\circ g$.
	
	\begin{defi}
		The Lipschitz saturation of $A$ in $B$ relative to $R \overset{\tau}{\rightarrow} A \overset{g}{\rightarrow} B$ is the set
		\begin{align*}
			A^*_{B,R}:=A^* := \left\{x \in B \mid \Delta(x) \in \overline{\ker \varphi}\right\}.
		\end{align*}
	\end{defi}
	If $A^* = g(A)$ then $A$ is said to be saturated in $B$. One important particular case is when $A$ is an $R$-subalgebra of $B$, taking $g$ as the inclusion. Let us see the first properties of the Lipschitz saturation.
	
	\begin{prop}[\cite{lipman1}]\label{prop_satsubanel}
		$A^*$ is an $R$-subalgebra of $B$ which contains $g(A)$.
	\end{prop}
	
	\begin{proof}
		
		Clearly $g(A)\sub A^*\sub B$. Let $x,y \in A^*$ and $r\in R$. So, $\Delta(x),\Delta(y)\in\overline{\ker\varphi}$. Then: $$\Delta(x+y)=\Delta(x)+\Delta(y)\in\overline{\ker\varphi}.$$
		
		Thus, $x+y\in A^*$. For the product, notice that $$\Delta(xy)=(x\ten_R1)\Delta(y)+(1\ten_Ry)\Delta(x)\in\overline{\ker\varphi}.$$
		
		Hence $xy\in A^*$. Finally, $$\Delta(rx)=r\Delta(x)=g(\tau(r))\Delta(x)\in\overline{\ker\varphi},$$
		
		\noindent and therefore, $rx\in A^*$.
	\end{proof}
	
	\begin{prop}[\cite{lipman1}]\label{202309051604}
		Let $A$ be an $R$-subalgebra of $B$ and let $C$ be an $R$-subalgebra of $B$ containing $A$ as an $R$-subalgebra. \[\begin{tikzcd}
			R & A & C & B
			\arrow["\tau", from=1-1, to=1-2]
			\arrow[hook, from=1-2, to=1-3]
			\arrow[hook, from=1-3, to=1-4]
			\arrow["g"', curve={height=12pt}, hook, from=1-2, to=1-4]
		\end{tikzcd}\]

		If $C^*=C^*_{B,R}$ then $A^* \subseteq C^*$.
	\end{prop}
	\begin{proof}
		Consider the diagram
		\begin{center}
			\begin{tikzcd}
				{B\underset{R}{\otimes}B} && {B\underset{A}{\otimes}B} \\
				& {B\underset{C}{\otimes}B}
				\arrow["\varphi", from=1-1, to=1-3]
				\arrow["{\varphi_C}"', from=1-1, to=2-2]
				\arrow["\lambda", from=1-3, to=2-2]
			\end{tikzcd}
		\end{center}
		\noindent where $\varphi_C$ e $\lambda$ are the canonical morphisms. It is easy to see this diagram commutes, so $\ker \varphi \subseteq \ker \varphi_C$, which implies $\overline{\ker\varphi}\sub\overline{\ker\varphi_C}$. Therefore, $A^*\sub C^*$.	
	\end{proof}
	
	\begin{cor}
		If $A$ is an $R$-subalgebra of $B$ then: 
		
		\begin{enumerate}
			\item $A \subseteq A^*$.
			
			\item $(A^*)^* = A^*$.
		\end{enumerate}
	\end{cor}
	
	\begin{proof}
		(1) In this case $g$ is the inclusion map, so $A=g(A)\sub A^*$.
		
		(2) Using (1), changing $A$ by $A^*$ (which is contained in $B$), we conclude that $A^* \subseteq (A^*)^*$. For the converse, consider the canonical morphism
		\begin{align*}
			\varphi^*:B\tens_{R} B \to B\otimes_{A^*}B.
		\end{align*}
		
		Notice that $\left\{\Delta(x)\mid x\in A^* \right\}$, and consequently $$\left<\left\{\Delta(x); x\in A^* \right\}\right> \subseteq \overline{\ker \varphi} \implies \ker \varphi^* \subseteq \overline{\ker \varphi}.$$
		
		\noindent Thus, $\overline{\ker \varphi^*} \subseteq \overline{\overline{\ker \varphi}} = \overline{\ker \varphi}$ which implies $\Delta^{-1}\prtt{\overline{\ker \varphi^*}} \subseteq \Delta^{-1}\prtt{\overline{\ker \varphi}}$. Hence, $$(A^*)^*\sub A^*.$$
	\end{proof}

	Here is the motivating example of relative saturation in the case of analytic complex varieties.

	\begin{example}\normalfont
		Following the approach of Pham-Teissier in \cite{phamt}, let $A$ be the local ring of an irreducible complex analytic space $X\sub\bC^n$ at the origin, and let $\overline{A}$ its normalization. Working with the canonical morphism $\varphi:\overline{A}\otimes_{\bC}\overline{A}\rightarrow \overline{A}\otimes_{A}\overline{A}$ (and with the analytic tensor product), Pham and Teissier defined the Lipschitz saturation of $A$ as the set $A^*$ of the elements $h\in \overline{A}$ such that $h\otimes 1-1\otimes h\in \overline{A}\otimes_{\bC}\overline{A}$ is in the integral closure of $\ker\varphi$. In our notation: $$A^*=\Delta^{-1}(\overline{\ker\varphi}),$$
		
		\noindent where $\Delta: \overline{A}\rightarrow \overline{A}\otimes_{\bC}\overline{A}$ is the diagonal map. 
		
		Now we present a connection between this notion and that of Lipschitz functions. If we choose generators $z_1,\hdots,z_n$ of the maximal ideal of the local ring then $\{\Delta(z_i)\}_{i=1}^n$ is a set of generators of $\ker\varphi$. One can choose $z_1,\hdots,z_n$ so that they are restrictions of coordinates on the ambient space $\bC^n$. Using the supremum criterion obtained by Jalabert and Teissier in \cite{lt}, for all $h\in\overline{A}$ one has:
		
		$$h\in A^*\iff \Delta(h)\in\overline{\ker\varphi}$$

		\noindent which is equivalent to the existence of some neighborhood $U$ of $(0,0)$ on $X\times X$ and a constant $C>0$ such that $$|\Delta(h)(z,z')|\leq C\sup\{|\Delta(z_i)(z,z')|\}_{i=1}^n, \forall (z,z')\in U.$$
		
		This last inequality is equivalent to $$|h(z)-h(z')|\leq C|z-z'|_{\mbox{\tiny{sup}}}, \forall (z,z')\in U,$$
		
		\noindent which is what is meant by the meromorphic function $h$ being Lipschitz at the origin in $X$.
	\end{example}
	
	Next, we see an algebraic version of the previous example.
	
	\begin{example}
		Let $k$ be an algebraically closed field, denote $\bA^n$ as the $n$-dimensional affine space over $k$, and let $V\sub \bA^n$ be an irreducible affine algebraic set. We recall some important notations: 
		
		\begin{itemize}
			\item $k[V]\rightarrow$ the coordinate ring of $V$ (which is a domain, once $V$ is irreducible);
			
			\item 	$k(V)\rightarrow$ field of fractions of $k[V]$;
			
			\item 
			$\overline{k[V]}\rightarrow$ integral closure of $k[V]$ in $k(V)$;
			
			\item For each $i\in\{1,\hdots,n\}$ we denote $\bx_i:V\sub\bA^n\rightarrow k$ as the projection onto the $i^{\tiny{\mbox{th}}}$-factor.
		\end{itemize}
		
		In this example we consider $$\begin{matrix}
			\Delta: & \overline{k[V]} & \rightarrow & \overline{k[V]}\ten_k\overline{k[V]}\\
			& f & \longmapsto & f\ten_k 1-1\ten_k f
		\end{matrix}$$ and the canonical map $$\begin{matrix}
			\varphi: & \overline{k[V]}\ten_k\overline{k[V]} & \longrightarrow & \overline{k[V]}\ten_{k[V]}\overline{k[V]}\\
		\end{matrix}.$$
		
		In this way, clearly $k[V]=k[\bx_1,\hdots,\bx_n]$. For all $i_1,\hdots,i_n\in\bZ_{\geq 0}$, using the Leibniz rule to get the equations $$\Delta(\bx_1^{i_1})=(\bx_1^{i_1-1}\ten_k 1)\Delta(\bx_1)+(1\ten_k\bx_1)\Delta(\bx_1^{i_1-1})$$
		
		$$\Delta(\bx_1^{i_1}\cdots \bx_{n}^{i_{n}})=((\bx_1^{i_1}\cdots \bx_{n-1}^{i_{n-1}})\ten_k 1)\Delta(\bx_n^{i_n})+(1\ten_k\bx_n^{i_n})\Delta(\bx_1^{i_1}\cdots \bx_{n-1}^{i_{n-1}})$$
		
		\noindent as inductive steps, one can see that $\Delta(\bx_1^{i_1}\cdots \bx_n^{i_n})\in\langle \Delta(\bx_i)\mid \forall i\in\{1,\hdots,n\}  \rangle$, for every $i_1,\hdots,i_n\in\bZ_{\geq 0}$. Since every $g\in k[V]$ is a polinomial at $\bx_1,\hdots,\bx_n$, one has $$\ker\varphi=\langle \Delta(\bx_i)\mid i\in\{1,\hdots,n\} \rangle.$$

		It is known that $k[V]\ten_kk[V]\cong k[V\times V]$ in a such way that for every $f,g\in k[V]$, the tensor product $f\ten_k g$ can be identified as the map $$\begin{matrix}
			f\ten_k g: & V\times V & \longrightarrow & k\\
			& (x,y)      &\longmapsto     & f(x)g(y)
		\end{matrix}.$$
		
		Thus, for each $i\in\{1,\hdots,n\}$ we have $$(\bx_i\ten_k1-1\ten_k\bx_i)(x_1,\hdots,x_n,y_1,\hdots,y_n)=x_i-y_i.$$
		
		Besides, we can conclude the Lipschitz saturation of $k[V]$ on $\overline{k[V]}$ is $$k[V]^*=\{f\in\overline{k[V]}\mid f\ten_k1-1\ten_kf\in\overline{\langle \bx_i\ten_k1-1\ten_k\bx_i \rangle_{i=1}^n}\}.$$		
	\end{example}

	\section{Some categorical results on the Lipschitz saturation}
	
	\begin{lema}\label{202305101554}
		Suppose that 
		\[\begin{tikzcd}
			A & B \\
			{A'} & {B'}
			\arrow["\alpha", from=1-1, to=1-2]
			\arrow["{\alpha'}", from=2-1, to=2-2]
			\arrow["\psi", from=1-2, to=2-2]
			\arrow["\phi"', from=1-1, to=2-1]
		\end{tikzcd}\]
		\noindent is a commutative diagram of ring morphisms. Then $$\phi(\overline{\ker\alpha})\sub\overline{\ker\alpha'}.$$ 
		
		If $\phi$ is surjective and $\psi$ is injective then the equality holds.
	\end{lema}
	
	\begin{proof}
		Since $\psi\circ\alpha=\alpha'\circ\phi$ then $\phi(\ker\alpha)\sub\ker\alpha'$. Let $v\in\phi(\overline{\ker\alpha})$. So there exists $u\in\overline{\ker\alpha}$ such that $v=\phi(u)$. Further, there are $a_1,\hdots,a_r\in A$ such that $$a_r+a_{r-1}u+\cdots+a_1u^{r-1}+u^r=0, \eqno{(\star)}$$
		
		\noindent where $a_i\in(\ker\alpha)^i, \forall i\in\{1,\hdots,r\}$. Thus, $b_i:=\phi(a_i)\in (\ker\alpha')^i$ and applying $\phi$ on $(\star)$ we obtain $$b_r+b_{r-1}v+\cdots+b_1v^{r-1}+v^r=0.$$
		
		\noindent Therefore, $v\in\overline{\ker\alpha'}$. 
		
		In the case where $\phi$ is surjective we have $\alpha'(A')\sub \psi(B)$, and if $\psi$ is injective then \[\begin{tikzcd}
			A & B \\
			{A'} & {\psi(B)}
			\arrow["\alpha", from=1-1, to=1-2]
			\arrow["{\tilde{\alpha'}}", from=2-1, to=2-2]
			\arrow["\tilde{\psi}", from=1-2, to=2-2]
			\arrow["\phi"', from=1-1, to=2-1]
		\end{tikzcd}\]
		
		\noindent is a commutative diagram, where $\tilde{\psi}$ is an isomorphism. Analogously, $\phi^{-1}(\overline{\ker\tilde{\alpha'}})\sub\overline{\ker\alpha}$ and therefore $\overline{\ker\alpha'}\sub\phi(\overline{\ker\alpha})$.
	\end{proof}

	\begin{prop} [\cite{lipman1}] \label{prop3}
		Suppose that
		\[\begin{tikzcd}
			R & A & B \\
			{R'} & {A'} & {B'}
			\arrow["\tau", from=1-1, to=1-2]
			\arrow["g", from=1-2, to=1-3]
			\arrow["{\tau'}", from=2-1, to=2-2]
			\arrow["{g'}", from=2-2, to=2-3]
			\arrow["{f_R}"', from=1-1, to=2-1]
			\arrow["{f}", from=1-3, to=2-3]
			\arrow["{f_A}", from=1-2, to=2-2]
		\end{tikzcd}\]
		
		\noindent is a commutative diagram of ring morphisms. Then
		$$f\prtt{A^*_{B,R}} \subseteq \prtt{A'}^*_{B',R'}.$$
	\end{prop}
	
	\begin{proof}
		Consider the canonical morphism $\varphi':B'\tens_{R'}B'\rightarrow B'\tens_{A'}B'$ and \begin{align*}
			\Delta': B'&\longrightarrow B'\tens_{R'}B' \\ x' &\longmapsto x'\tens_{R'}1-1\tens_{R'}x'.
		\end{align*}

		The universal property of the tensor product guarantees the existence of morphisms $\phi$ and $\psi$ which maps 
		\[\begin{tikzcd}
			{b_1\underset{R}{\otimes}b_2} & {f(b_1)\underset{R'}{\otimes}f(b_2)} \\
			{b_1\underset{A}{\otimes}b_2} & {f(b_1)\underset{A'}{\otimes}f(b_2)}
			\arrow["\phi", maps to, from=1-1, to=1-2]
			\arrow["\psi", maps to, from=2-1, to=2-2]
		\end{tikzcd}\]
		
		\noindent and consequently, the diagram 
		
		\[\begin{tikzcd}
			B && {B\underset{R}{\otimes}B} && {B\underset{A}{\otimes}B} \\
			\\
			{B'} && {B'\underset{R'}{\otimes}B'} && {B'\underset{A'}{\otimes}B'}
			\arrow["\Delta", from=1-1, to=1-3]
			\arrow["\varphi", from=1-3, to=1-5]
			\arrow["{\Delta'}"', from=3-1, to=3-3]
			\arrow["{\varphi'}"', from=3-3, to=3-5]
			\arrow["f"', from=1-1, to=3-1]
			\arrow["\phi"', from=1-3, to=3-3]
			\arrow["\psi"', from=1-5, to=3-5]
		\end{tikzcd}\] \noindent commutes. Since $A^*_{B,R}=\Delta^{-1}(\overline{\ker\varphi})$ then $\Delta(A^*_{B,R})\sub\overline{\ker\varphi}\implies \phi(\Delta(A^*_{B,R}))\sub\phi(\overline{\ker\varphi})$. Now, Lemma \ref{202305101554} and the commutativeness of the above diagram imply that $\Delta'(f(A^*_{B,R}))\sub \overline{\ker\alpha'}$, and therefore $$f(A^*_{B,R})\sub\Delta'^{-1}(\overline{\ker\alpha'})=(A')^*_{B',R'}.$$ 
	\end{proof}

	Here we point out that Proposition \ref{202309051604} proved by Lipman is true even in the case where the sequence of rings $A\rightarrow C\rightarrow B$ is not necessarily a chain of subalgebras. 
	
	\begin{prop}
		If 
		$\begin{tikzcd}
			R & A & C & B
			\arrow["\tau", from=1-1, to=1-2]
			\arrow["\lambda", from=1-2, to=1-3]
			\arrow["{g_C}", from=1-3, to=1-4]
			\arrow["g"', curve={height=12pt}, from=1-2, to=1-4]
		\end{tikzcd}$ is a sequence of ring morphisms then $$A^*_{B,R}\sub C^*_{B,R}.$$
	\end{prop}
	
	\begin{proof}
		
		Indeed, the given sequence of ring morphisms induces the commutative diagram 
		$$\begin{tikzcd}
			R & A & B \\
			R & C & B
			\arrow["\tau", from=1-1, to=1-2]
			\arrow["g", from=1-2, to=1-3]
			\arrow["{\textrm{id}_B}", from=1-3, to=2-3]
			\arrow["{\textrm{id}_R}"', from=1-1, to=2-1]
			\arrow["{\tau_C}"', from=2-1, to=2-2]
			\arrow["\lambda" ', from=1-2, to=2-2]
			\arrow["{g_C}"', from=2-2, to=2-3]
		\end{tikzcd}.$$ Now it suffices to apply Proposition \ref{prop3} to get $A^*_{B,R}\sub C^*_{B,R}$.	
	\end{proof}

	Let us fix some notation before the next result. Let $I$ be a directed set. We denote a direct system over $I$ as $(A_{\bullet},\nu)$, which means that we have a collection $\{A_i\}_{i\in I}$ of rings and for every $i\leq j\in I$ we have a ring morphism $\nu_{ij}:A_i\rightarrow A_j$ satisfying the known conditions for a direct system.
	
	It is well known that the category of rings is complete and cocomplete. In particular, $(A_{\bullet},\nu)$ admits a direct limit $\varinjlim\limits_{I}A_{\bullet}=(A,\alpha)$, i.e, there exists a collection of ring morphisms $\{\alpha_i:A_i\rightarrow A \}$ satisfying the known conditions for a direct limit. This collection we denote as $\alpha_{\bullet}:A_{\bullet}\rightarrow A$.
	
	In the case of rings we can construct explicitly $\varinjlim\limits_{I}A_{\bullet}$ as the quotient $\dfrac{\bigoplus\limits_{i\in I}A_i}{S_{\nu}}$, where $S_{\nu}$ is the $\mathbb{Z}$-submodule of $\bigoplus\limits_{i\in I}A_i$ generated by sequences on the form $(x_i)_{i\in I}$ where $\nu_{ij}(x_i)=x_j$, for all $i\leq j\in I$. This limit can be endowed with a ring structure induced by $A_i, i\in I$, such that the canonical maps $\alpha_i:A_i\rightarrow A$ ($i\in I$) (which are formed by composing inclusion with the quotient map) satisfies the universal property for the direct limit.

	\begin{prop} [\cite{lipman1}]\label{202305112023}
		Let $(R_{\bullet},\mu)\overset{\tau_{\bullet}}{\rightarrow}(A_{\bullet},\nu)\overset{g_{\bullet}}{\rightarrow}(B_{\bullet},\theta)$ be a sequence of morphisms of direct systems of rings over a directed set $I$. Suppose that $\rho_{\bullet}:R_{\bullet}\rightarrow R$, $\alpha_{\bullet}:A_{\bullet}\rightarrow A$ and $\beta_{\bullet}:B_{\bullet}\rightarrow B$ are direct limits of $R_{\bullet}$, $A_{\bullet}$ and $B_{\bullet}$, respectively. It is well known that there exist ring morphisms $\tau:R\rightarrow A$ and $g:A\rightarrow B$ such that 
		\[\begin{tikzcd}
			{R_i} & {A_i} & {B_i} \\
			R & A & B
			\arrow["{\tau_i}", from=1-1, to=1-2]
			\arrow["{g_i}", from=1-2, to=1-3]
			\arrow["\tau", from=2-1, to=2-2]
			\arrow["g", from=2-2, to=2-3]
			\arrow["{\rho_i}"', from=1-1, to=2-1]
			\arrow["{\alpha_i}"', from=1-2, to=2-2]
			\arrow["{\beta_i}", from=1-3, to=2-3]
		\end{tikzcd}\]\noindent is a commutative diagram, $\forall i\in I$.
		
		\begin{enumerate}
			\item For all $i\leq j\in I$ the map $$\begin{matrix}
				\theta^*_{ij}: & (A_i)^*_{B_i,R_i} & \longrightarrow & (A_j)^*_{B_j,R_j}\\
				& z & \longmapsto & \theta_{ij}(z) 
			\end{matrix}$$\noindent is a well-defined morphism of rings.
			
			\item $((A_{\bullet})^*_{B_{\bullet}, R_{\bullet}},\theta^*)$ is a direct system of rings over $I$.
			
			\item $\varinjlim\limits_{I}(A_{\bullet})^*_{B_{\bullet}, R_{\bullet}}\cong A^*_{B,R}$.
		\end{enumerate}
		
	\end{prop}
	
	\begin{proof}
		(1) Since $\tau_{\bullet}$ and $g_{\bullet}$ are morphisms of direct systems, for $i\leq j\in I$ we have the following commutative diagram: 
		\[\begin{tikzcd}
			{R_i} & {A_i} & {B_i} \\
			{R_j} & {A_j} & {B_j}
			\arrow["{\tau_i}", from=1-1, to=1-2]
			\arrow["{g_i}", from=1-2, to=1-3]
			\arrow["{\tau_j}", from=2-1, to=2-2]
			\arrow["{g_j}", from=2-2, to=2-3]
			\arrow["{\mu_{ij}}"', from=1-1, to=2-1]
			\arrow["{\nu_{ij}}"', from=1-2, to=2-2]
			\arrow["{\theta_{ij}}", from=1-3, to=2-3]
		\end{tikzcd}\]
		
		\noindent Now, Proposition \ref{prop3} implies $\theta_{ij}((A_i)^*_{B_i,R_i})\sub (A_j)^*_{B_j,R_j}$. Hence, $\theta_{ij}^*$ is well-defined, and clearly it is a ring morphism.
		
		(2) (i) For each $i\in I$, since $(B_{\bullet},\theta)$ is a direct system then $\theta_{ii}=\id_{B_i}$. Hence, $\theta_{ii}^*=\id_{(A_i)^*_{B_i,R_i}}$.
		
		\noindent (ii) Let $i\leq j\leq k\in I$. Since $\theta_{ik}=\theta_{jk}\circ\theta_{ij}$ then $\theta^*_{ik}=\theta^*_{jk}\circ\theta^*_{ij}$.
		
		(3) Using the previous notation it is straightforward to check that the map $$\begin{matrix}
			\psi: & \varinjlim\limits_{I}(A_i)^*_{B_i,R_i} & \longrightarrow & A^*_{B, R}\\
			& (x_i)_{i\in I}+S_{\theta^*} & \longmapsto & (\beta_i(x_i))_{i\in I}+S_{\theta}
		\end{matrix}$$ \noindent is an isomorphism.

	\end{proof}

	\begin{prop} [\cite{lipman1}] Let $g_i:A_i \rightarrow B_i$ be $R$-algebra morphisms and let $g:A \rightarrow B$ be the direct product of those maps. That is,
		\begin{align*}
			A = \prtt{\prod\limits_{i=1}^{n}A_i} \overset{g = \prod\limits g_i}{\longrightarrow} \prtt{\prod\limits_{i=1}^{n}B_i} = B.
		\end{align*}
		Então,
		\begin{align*}
			A_{B,R}^* = \prod\limits_{i=1}^{n}\prtt{A_i}_{B_i,R}^*.
		\end{align*}
	\end{prop}
	\begin{proof}
		For each $i\in \{1,\hdots,n\}$ consider the canonical morphism $\varphi_i: B_i\ten_R B_i \rightarrow B_i\ten_{A_i} B_i$ and the diagonal $\Delta_i: B_i\rightarrow B_i\ten_R B_i$. Let $\tilde{\varphi}$ be the composition of the morphisms 
		\[\begin{tikzcd}
			{\prod\limits_{i,j=1}^n(B_i\underset{R}{\otimes}B_j)} && {\prod\limits_{i=1}^n(B_i\underset{R}{\otimes}B_i)} && {\prod\limits_{i=1}^n(B_i\underset{A_i}{\otimes}B_i)}
			\arrow["{\textrm{\scriptsize{p=projection}}}", from=1-1, to=1-3]
			\arrow["{\prod\varphi_i}", from=1-3, to=1-5]
		\end{tikzcd}\]
		
		Thus, $\ker\tilde{\varphi}=\ker(\prod\varphi_i)\times \left(\prod\limits_{i\neq j=1}^n(B_i\underset{R}{\otimes}B_j)\right)$. Once the kernel and integral closure for ideals commute with a finite direct product, one has $$\overline{\ker\tilde{\varphi}}=\left(\prod\limits_{i=1}^n\overline{\ker\varphi_i}\right)\times \prod\limits_{i\neq j=1}^n(B_i\underset{R}{\otimes}B_j)=p^{-1}\left(\prod\limits_{i=1}^n\overline{\ker\varphi_i}\right).$$

		On the other hand, we have a commutative diagram 
		\[\begin{tikzcd}
			{\prod\limits_{i,j=1}^n(B_i\underset{R}{\otimes}B_j)} && {\prod\limits_{i=1}^n(B_i\underset{A_i}{\otimes}B_i)} \\
			\\
			{B\underset{R}{\otimes}B} && {B\underset{A}{\otimes}B}
			\arrow["{\tilde{\varphi}}", from=1-1, to=1-3]
			\arrow["\varphi", from=3-1, to=3-3]
			\arrow["\phi"', from=1-1, to=3-1]
			\arrow["\psi", from=1-3, to=3-3]
		\end{tikzcd}\]
		
		\noindent where $\phi$ is an isomorphism, $\psi$ is a canonical injective morphism and $p\circ\phi^{-1}\circ\Delta=\prod\Delta_i$. Hence, $\phi(\overline{\ker\tilde{\varphi}})=\overline{\ker\varphi}$. Therefore: 
		
		$$A^*_{B,R}=\Delta^{-1}(\overline{\ker\varphi})=\Delta^{-1}(\phi(\overline{\ker\tilde{\varphi}}))=\Delta^{-1}\left(\phi\left(p^{-1}\left(\prod\limits_{i=1}^n\overline{\ker\varphi_i}\right)\right)\right)$$
		
		$$=(p\circ\phi^{-1}\circ\Delta)^{-1}\left(\prod\limits_{i=1}^n\overline{\ker\varphi_i}\right)=\left(\prod\Delta_i\right)\left(\prod\limits_{i=1}^n\overline{\ker\varphi_i}\right)=\prod\limits_{i=1}^n\Delta_i^{-1}(\overline{\ker\varphi_i})=\prod\limits_{i=1}^n(A_i)^{*}_{B_i,R}.$$
	\end{proof}
	
	In the next proposition, we prove the base change for a flat algebra partially preserves Lipschitz saturation.
	
	\begin{prop}\label{202305171730}
		Let $R\overset{\tau}{\rightarrow} A \overset{g}{\rightarrow}B$ be a sequence of ring morphisms, let $R'$ be a flat $R$-algebra and consider the induced sequence 
		\[\begin{tikzcd}
			{R'\cong R\underset{R}{\otimes}R'} && {A\underset{R}{\otimes}R'} && {B\underset{R}{\otimes}R'}
			\arrow["{\tau'}", from=1-1, to=1-3]
			\arrow["{g'=g\underset{R}{\otimes}\textrm{id}_{R'}}", from=1-3, to=1-5]
		\end{tikzcd}\] Then: $$A^*_{B,R}\ten_RR'\sub(A\ten_RR')^*_{B\ten_RR',R'}.$$
	\end{prop}
	
	\begin{proof}
		Denote $A':=A\ten_RR'$, $B':=B\ten_RR'$, and consider the canonical maps 
		\[\begin{tikzcd}
			{B'} && {B'\underset{R'}{\otimes}B'} && {B'\underset{A'}{\otimes}B'}
			\arrow["{\Delta'}", from=1-1, to=1-3]
			\arrow["{\varphi'}", from=1-3, to=1-5]
		\end{tikzcd}\]
		
		We have $$B'\ten_{R'}B'=(B\ten_RR')\ten_{R'}(B\ten_RR')\cong B\ten_R(R'\ten_{R'}B)\ten_RR'$$$$\cong(B\ten_RB)\ten_RR'\cong B\ten_RB',$$
		
		\noindent and $(B\ten_AB)\ten_RR'\cong B\ten_A(B\ten_RR')=B\ten_AB'$ which maps canonically to $B'\ten_A B'$, and then to $B'\ten_{A'}B'$. Thus, there exists a commutative diagram 
		\[\begin{tikzcd}
			&& {B'} \\
			{B\underset{R}{\otimes}B'} && {(B\underset{R}{\otimes}B)\underset{R}{\otimes}R'} && {(B\underset{A}{\otimes}B)\underset{R}{\otimes}R'} \\
			{B'\underset{R'}{\otimes}B'} &&&& {B'\underset{A'}{\otimes}B'}
			\arrow["{\varphi'}", from=3-1, to=3-5]
			\arrow["\gamma", from=2-1, to=2-3]
			\arrow["{\varphi\underset{R}{\otimes}\textrm{id}_{R'}}", from=2-3, to=2-5]
			\arrow["{\Delta\underset{R}{\otimes}\textrm{id}_{R'}}", from=1-3, to=2-3]
			\arrow["\phi", from=2-1, to=3-1]
			\arrow["\psi", from=2-5, to=3-5]
			\arrow["{\phi^{-1}\circ\Delta'}"', from=1-3, to=2-1]
		\end{tikzcd}\]\noindent where $\gamma$ and $\phi$ are ring isomorphisms. 
		
		Let $x\in A^*_{B,R}$ and $y\in R'$. So, $\Delta(x)\in\overline{\ker\varphi}$ and consequently $$\gamma\circ\phi^{-1}\circ\Delta'(x\ten_Ry)=(\Delta\ten_R\id_{R'})(x\ten_Ry))=\Delta(x)\ten_Ry\in \overline{\ker\varphi}\ten_RR'.$$
		
		Since $\ker\varphi$ is an ideal of $B\ten_RB$ then $\overline{\ker\varphi}\ten_RR'\sub \overline{\ker\varphi\ten_RR'}$, which is contained in $\overline{\ker(\varphi\ten_R\id_{R'})}$. In particular, $\Delta'(x\ten_Ry)\in\phi(\gamma^{-1}(\overline{\ker(\varphi\ten_R\id_{R'})}))$. Once $\gamma$ is an isomorphism, we have $$\gamma^{-1}(\overline{\ker(\varphi\ten_R\id_{R'})})=\overline{\gamma^{-1}(\ker(\varphi\ten_R\id_{R'}))}\sub\overline{\ker((\varphi\ten_R\id_{R'})\circ\gamma)}$$$$\sub\overline{\ker(\psi\circ(\varphi\ten_R\id_{R'})\circ\gamma)}=\overline{\ker(\varphi'\circ\phi)}.$$
		
		Hence, $\Delta'(x\ten_Ry)\in\phi(\overline{\ker(\varphi'\circ\phi)})$. Since $\phi$ is an isomorphism then $$\phi(\overline{\ker(\varphi'\circ\phi)})=\overline{\phi(\ker(\varphi'\circ\phi))}=\overline{\ker\varphi'}.$$
		
		Therefore, $\Delta'(x\ten_Ry)\in\overline{\ker\varphi'}\implies x\ten_Ry \in (A')^*_{B',R'}$, which finishes the proof.
		
	\end{proof}

	The next corollary is in \cite{lipman1}, and here we get it as a straightforward consequence of our previous proposition.
	
	\begin{cor}[Faithfully flat descent]
		Under the notation of \ref{202305171730}, assume that $R'$ is a faithfully flat $R$-algebra. If $A\ten_RR'$ is $R'$-saturated in $B\ten_RR'$ then $A$ is $R$-saturated on $B$.
	\end{cor}	
	
	\begin{proof}
		Indeed, by hypothesis $(A')^*_{B', R'}=g'(A')=(g\ten_R\id_{R'})(A\ten_RR')=g(A)\ten_RR'$. Proposition \ref{202305171730} implies that $A^*_{B,R}\ten_RR'\sub g(A)\ten_RR'$, and so $$A^*_{B,R}\ten_RR'= g(A)\ten_RR'.$$
		
		Since $R'$ is $R$-flat then $0=\dfrac{A^*_{B,R}\ten_RR'}{g(A)\ten_RR'}=\dfrac{A^*_{B,R}}{g(A)}\ten_RR'$ and the faithfulness of $R'$ implies that $\dfrac{A^*_{B,R}}{g(A)}=0$, i.e, $A^*_{B,R}=g(A)$.
	\end{proof}

	\section{Relative Lipschitz saturation and radicial algebras}

	Let us recall some important definitions before continuing.
	
	\begin{defi}
	Let $h:S\rightarrow T$ be a ring morphism. We say that $T$ is a \textbf{radicial} $S$-algebra if:
	
	\begin{enumerate}
		\item The induced map on spectra $\Spec(h):\Spec T\rightarrow \Spec S$ is injective;
		
		\item For every $\fq\in\spec S$ the embedding $h_{\fq}:\Frac\left(\dfrac{S}{h^{-1}(\fq)}\right)\hookrightarrow \Frac\left(\dfrac{T}{\fq}\right)$ induced by $h$ is purely inseparable.\footnote{$\Frac$ stands for the field of fractions.}
	\end{enumerate}
	
	In this case, we also say that $h$ is a radicial ring morphism. 
	\end{defi}
	
	Recall that if $K$ is a field and $T$ is a ring we can identify $\textrm{Mor}_{\textrm{\textbf{\tiny{Schemes}}}}(\textrm{Spec} K, \textrm{Spec} T)$ as the set $$\{(\fq,\beta)\mid \fq\in\Spec T\mbox{ and }\beta:\Frac\left(\frac{T}{\fq}\right)\rightarrow K\mbox{ is a ring morphism}\}.$$

	We recall some properties of radicial algebras.
	
	\begin{teo}\label{202305262023} The following conditions are equivalent:
	
	\begin{enumerate}
		\item [(a)] $h:S\rightarrow T$ is a radicial ring morphism;
		
		\item [(b)] Any two distinct morphisms of $T$ into a field have distinct composition with $h$;
		
		\item [(c)]  The kernel of the canonical morphism $\gamma: T\ten_ST\rightarrow T$ is a nil ideal of $T\ten_ST$;
		
		\item [(d)] $t\ten_S 1-1\ten_St$ is nilpotent in $T\ten_ST$, $\forall t\in T$.
	\end{enumerate} 
	
	\end{teo}
	
	\begin{proof}
	Let $\hat{h}:\Spec T\rightarrow \Spec S$ be the morphism of schemes induced by $h$ and let $K$ be a field. Then we have a canonical commutative diagram of maps 
	\[\begin{tikzcd}
		{\textrm{Mor}_{\textrm{\textbf{\tiny{Rings}}}}(T,K)} && {\textrm{Mor}_{\textrm{\textbf{\tiny{Rings}}}}(S,K)} \\
		{\textrm{Mor}_{\textrm{\textbf{\tiny{Schemes}}}}(\textrm{Spec} K, \textrm{Spec} T)} && {\textrm{Mor}_{\textrm{\textbf{\tiny{Schemes}}}}(\textrm{Spec} K, \textrm{Spec} S)}
		\arrow["{h^*}", from=1-1, to=1-3]
		\arrow["\Psi"', from=2-1, to=2-3]
		\arrow["{\lambda_T}"', from=1-1, to=2-1]
		\arrow["{\lambda_S}", from=1-3, to=2-3]
	\end{tikzcd}\]
	
	\noindent where \begin{itemize}
		\item $\lambda_T$ and $\lambda_S$ are canonical bijections;
		
		\item $h^*(\delta)=\delta\circ h,\forall \delta\in \textrm{Mor}_{\textrm{\textbf{\tiny{Rings}}}}(T,K)$;
		
		\item $\Psi(\alpha)=\hat{h}\circ\alpha, \forall \alpha\in \textrm{Mor}_{\textrm{\textbf{\tiny{Schemes}}}}(\textrm{Spec} K, \textrm{Spec} T)$;
		
		\item $\Psi(\fq,\beta)=(h^{-1}(\fq),\beta\circ h_{\fq})$, $\forall (\fq,\beta)\in \textrm{Mor}_{\textrm{\textbf{\tiny{Schemes}}}}(\textrm{Spec} K, \textrm{Spec} T)$.
	\end{itemize}
	
	Now, it is clear that $h$ is radicial if and only if $\Psi$ is injective, and since the above diagram commutes and $\lambda_T$ and $\lambda_S$ are bijections, then $\Psi$ is injective if and only if $h^*$ is injective. 
	
	Hence, the equivalence $(a)\iff(b)$ is proved.
	
	The equivalence $(c)\iff(d)$ holds because the kernel of the canonical map $T\ten_ST\rightarrow T$ is generated by the elements of the form $t\ten_S 1-1\ten_St$ is nilpotent in $T\ten_ST$, $\forall t\in T$.
	
	(b)$\iff$(c) In this case it is useful to work with the diagonal map $$\mathbf{\Delta}_{X/S}:X\rightarrow X\times_{\Spec S}X, $$
	
	\noindent where $X:=\Spec T$, which corresponds to the ring morphism $\gamma:T\ten_ST\rightarrow T$. Thus, (b) is equivalent to $\mathbf{\Delta_{X/S}}$ to be surjective (see   \cite{EGA}, Prop. 3.7.1).
	\end{proof}

	\begin{defi}[Unramified algebra]
	Let $T$ be a $S$-algebra and $\gamma:T\ten_ST\rightarrow T$ be the canonical morphism which takes $u\ten_Sv\mapsto uv, \forall u,v\in T$. We say that $T$ is an \textbf{unramified} $S$-algebra if the following conditions hold: 
	
	\begin{enumerate}
		\item $\ker\gamma$ is a finitely generated ideal of $T\ten_ST$;
		
		\item $(\ker\gamma)^2=\ker\gamma$.
	\end{enumerate}
	\end{defi}
	
	Using the determinant trick it is easy to see that condition (2) implies that there exists $e\in \ker\gamma$ such that $\ker\gamma$ is the principal ideal generated by $e$, and $e^2=e$.
	
	\begin{lema}\label{202305261559}
	Let $h:S\rightarrow T$ be a surjective ring morphism, $x\in S$, and let $I$ be an ideal of $S$. If $h(x)\in\overline{h(I)}$ then $x$ is integral over $I\mod \ker h$, i.e, $x+\ker h\in\overline{\left(\dfrac{I}{\ker h}\right)}$.   
	\end{lema}
	
	\begin{proof}
	By hypothesis there exist $b_i=h(a_i)$, $a_i\in I^i, \forall i\in\{1,\hdots,n\}$, such that $$h(x^n+a_1x^{n-1}+\cdots+a_{n-1}x+a_n)=0.$$
	
	If we denote $\bar{u}:=u+\ker h$, $\forall u\in S$, the last equation implies that $$\bar{x}^n+\overline{a_1}\bar{x}^{n-1}+\cdots+\overline{a_{n-1}}\bar{x}+\overline{a_n}=\bar{0},$$
	
	\noindent with $\overline{a_i}\in\left(\dfrac{I}{\ker h}\right)^i$, $\forall i\in \{1,\hdots,n\}$, which finishes the proof. 
	\end{proof}
	
	Before we continue, let us prove a useful lemma to proceed with this section.
	
	\begin{lema}\label{202305261817}
	Let $R$ be an ideal and $I, J_1,\hdots,J_n$ ideals of $R$ where the product $J_1\cdots J_n$ is a nil ideal of $R$. If $x\in R$ is integral over $I\mod J_i$, $\forall i\in\{1,\hdots,n\}$ then $x\in\overline{I}$.
	\end{lema}
	
	\begin{proof}
	For each $i\in \{1,\hdots,n\}$ there exists a monic polynomial $p_i\in R[X]$, with suitable coefficients related to the integral dependence over $I\mod J_i$, such that $p_i(x)\in J_i$. Setting $p:=p_1\cdots p_n$, one has $p(x)\in J_1\cdots J_n\sub\sqrt{(0)}$ which ensures the existence of an $r\in\bN$ such that $(p(x))^r=0$. Therefore, $x\in \overline{I}$.
	\end{proof}
	
	In the following result, Lipman demonstrated that a radicial base change does not alter the relative Lipschitz saturation.
	
	\begin{prop}[\cite{lipman1}]
	Consider the diagram of ring morphisms 
	\[\begin{tikzcd}
		R & {R'} & A & B
		\arrow["g", from=1-3, to=1-4]
		\arrow["{\tau'}", from=1-2, to=1-3]
		\arrow["\lambda", from=1-1, to=1-2]
		\arrow["{\tau:=\tau'\circ\lambda}"', curve={height=18pt}, from=1-1, to=1-3]
	\end{tikzcd}\]
	\end{prop}
	Then: 
	\begin{enumerate}
	\item [a)] $A^*_{B,R}\sub A^*_{B,R'}$;
	
	\item [b)] If $R'$ is a radicial or unramified $R$-algebra then $A^*_{B,R}=A^*_{B,R'}$.
	\end{enumerate}
	
	\begin{proof}
	(a) By hypothesis the diagram 
	\[\begin{tikzcd}
		R & A & B \\
		{R'} & A & B
		\arrow["\lambda"', from=1-1, to=2-1]
		\arrow["{\textrm{id}_A}"', from=1-2, to=2-2]
		\arrow["{\textrm{id}_B}", from=1-3, to=2-3]
		\arrow["\tau", from=1-1, to=1-2]
		\arrow["g", from=1-2, to=1-3]
		\arrow["{\tau'}"', from=2-1, to=2-2]
		\arrow["g"', from=2-2, to=2-3]
	\end{tikzcd}\] \noindent is commutative. Now Proposition \ref{prop3} implies that $A^*_{B,R}=\id_B(A^*_{B,R})\sub A^*_{B,R'}$.
	
	(b) Let $\gamma:R'\ten_RR'\rightarrow R'$ be the canonical morphism. Let us prove that in either case there exists $e'\in \ker\gamma$ such that $e'^2=e'$ and $\ker\gamma\sub \sqrt{(e')}$.
	
	(i) Assuming that $R'$ is a radicial $R$-algebra, the previous theorem says that $\ker\gamma$ is a nil ideal of $R'\ten_RR'$. In this case, take $e':=0$.
	
	(ii) Suppose that $R'$ is an unramified $R$-algebra. As we observed, there exists $e'\in \ker\gamma$ such that $e'^2=e'$ and $\ker\gamma=(e')$. In particular, $\ker\gamma\sub\sqrt{(e')}$, and our claim is proved.
	
	Now, before to prove another inclusion, let us prepare the way. Consider the commutative diagram 
	\[\begin{tikzcd}
		{R'\underset{R}{\otimes}R'} && {A\underset{R}{\otimes}A} && {B\underset{R}{\otimes}B} && {B\underset{A}{\otimes}B} \\
		&& B && {B\underset{R'}{\otimes}B}
		\arrow["{\tau'\underset{R}{\otimes}\tau'}", from=1-1, to=1-3]
		\arrow["{g\underset{R}{\otimes}g}", from=1-3, to=1-5]
		\arrow["\varphi", from=1-5, to=1-7]
		\arrow["{\Delta'}"', from=2-3, to=2-5]
		\arrow["{\varphi'}"', from=2-5, to=1-7]
		\arrow["\psi"', from=1-5, to=2-5]
		\arrow["\Delta"{pos=0.4}, curve={height=6pt}, from=2-3, to=1-5]
	\end{tikzcd}\]\noindent where $\psi$ is the canonical morphism, and set $\sigma:=(g\ten_Rg)\circ(\tau'\ten_R\tau')$. We already know that $\ker\psi$ is the ideal of $B\ten_RB$ by the image of $\Delta\circ g\circ\tau'$. It is straightforward to conclude that $$\Delta(g(\tau'(b)))=\sigma(b\ten_R1_{R'}-1_{R'}\ten_Rb), \forall b\in R',$$
	
	\noindent and since $\ker\gamma$ is generated by $\{b\ten_R1_{R'}-1_{R'}\ten_Rb\mid b\in R'\}$ then $\ker\psi$ is the ideal of $B\ten_RB$ generated by $\sigma(\ker\gamma)$. In particular, $e:=\sigma(e')\in\ker \psi$ and $e^2=e$. Further, since $\ker\gamma\sub\sqrt{(e')}$ then $\ker\psi\sub\sqrt{(e)}$. Setting $J$ as the ideal of $B\ten_RB$ generated by $1-e$. Thus, once $e(1-e)=0$, it follows that  $$(\ker\psi)J\sub\sqrt{(e)}(1-e)=(0),$$\noindent and consequently, $(\ker\psi)J$ is a nil ideal of $B\ten_RB$. Since $\psi$ is surjective and $\varphi=\varphi'\circ\psi$ then $\psi(\ker\varphi)=\ker\varphi'$. 
	
	Finally, let us check the inclusion $A^*_{B,R'}\sub A^*_{B,R}$. Taking any $x\in A^*_{B,R'}$, one has $$\psi(\Delta(x))=\Delta'(x)\in\overline{\ker\varphi'}=\overline{\psi(\ker\varphi)},$$
	
	\noindent and Lemma \ref{202305261559} implies that $\Delta(x)$ is integral over $\ker\varphi\mod \ker\psi$. Besides, since $e\in \ker\psi$ then $c:=\Delta(x)e\in \ker\psi$, and clearly $\ker\psi\sub\ker\varphi$, hence $c\in \ker\varphi$. Notice that $$\Delta(x)-c=\Delta(x)(1-e)\in J,$$
	
	\noindent which implies that $\Delta(x)$ is integral over $\ker\varphi\mod J$. Since $(\ker\varphi)J$ is a nil ideal, Lemma \ref{202305261817} ensures that $\Delta(x)\in\overline{\ker\varphi}$, and therefore $x\in A^*_{B,R}$.  
	\end{proof}
	
	Before presenting the proof of the main theorem of this work, let us establish some auxiliary results first.
	
	\begin{lema}\label{202305261918}
	Let $h:S\rightarrow T$ be a ring morphism, $x\in S$, and suppose that $\ker h$ is a nil ideal of $S$. If $h(x)\in\sqrt{(0_T)}$ then $x\in \sqrt{(0_S)}$.
	\end{lema}
	
	\begin{proof}
	By hypothesis there exists $r\in\bN$ such that $(h(x))^r=0_T$, i.e, $h(x^r)=0_T$. So, $x^r\in\ker h\sub\sqrt{(0_S)}$, which implies that $x\in\sqrt{\sqrt{(0_S)}}=\sqrt{0_S}$.
	\end{proof}
	
	In the next lemma, we observe that we do not need to require the kernel of $h$ to be a nil ideal in order to guarantee that $\ker(h\ten_Rh)$ is a nil ideal, as in \cite{lipman1}.
	
	\begin{lema}\label{202305241953}
	Let $h:S\rightarrow T$ be an integral morphism of $R$-algebras. 
	
	\begin{enumerate}
		\item [a)] $\ker(h\ten_R h)$ is a nil ideal of $S\ten_RS$;
		
		\item [b)] Suppose that $\ker h$ is a nil ideal of $S$. Then $\overline{I}=h^{-1}(\overline{IT})$, for every $I$ ideal of $S$.
	\end{enumerate}
	\end{lema}
	
	\begin{proof}
	(a) We want to show that $\ker(h\ten_Rh)\sub\bigcap\limits_{\fp\in\tiny{\mbox{Spec}(S\ten_RS)}}\fp$, so it suffices to show that the induced map $(h\ten_Rh)^\sharp:\spec(T\ten_RT)\rightarrow \spec(S\ten_RS)$ is surjective. 
	
	Let $\fp\in\spec(S\ten_RS)$. The domain $\dfrac{S\ten_RS}{\fp}$ can be embedded into an algebraically closed field $F$ such that the kernel of the composition $$S\ten_RS\overset{\mbox{\tiny{projection}}}{\longrightarrow}\dfrac{S\ten_RS}{\fp}\hookrightarrow F$$
	
	\noindent is $\fp$. Let $\alpha$ be this composition and let $\gamma_1,\gamma_2:S\rightarrow S\ten_RS$ be the canonical maps which takes $s\ten_R1\overset{\gamma_1}{\mapsfrom}s\overset{\gamma_2}{\mapsto}1\ten_Rs$, for all $s\in S$. Since $h$ is an integral morphism then there exist ring morphisms $\delta_1,\delta_2:T\rightarrow F$ for which the diagram 
	\[\begin{tikzcd}
		T \\
		S \\
		& {S\underset{R}{\otimes}S} && F \\
		S \\
		T
		\arrow["h"', from=2-1, to=1-1]
		\arrow["{\gamma_1}"', from=2-1, to=3-2]
		\arrow["{\gamma_2}", from=4-1, to=3-2]
		\arrow["h"', from=4-1, to=5-1]
		\arrow["\alpha", from=3-2, to=3-4]
		\arrow["{\delta_1}"', dashed, from=1-1, to=3-4]
		\arrow["{\delta_2}"', dashed, from=5-1, to=3-4]
	\end{tikzcd}\]\noindent commutes. The universal property of the tensor product ensures the existence of a unique ring morphism $\beta:T\ten_RT\rightarrow F$ such that $\beta(u\ten_Rv)=\delta_1(u)\delta_2(v)$, $\forall u,v\in T$. Defines $\fq:=\ker\beta\in\spec(T\ten_RT)$. Clearly $\alpha=\beta\circ(h\ten_Rh)$, and since $\fp=\ker\alpha$ then we conclude that $(h\ten_Rh)^{-1}(\fq)=\fp$. 
	
	(b) The persistence of the integral closure of ideals implies $h(\overline{I})\sub\overline{IT}$. Conversely, assume that $x\in h^{-1}(\overline{IT})$. Thus, $y:=h(x)$ is integral over $IT$, and then $yX$ is integral over $T[(IT)X]$, which is integral over $h(S)[h(I)X]$, once $h$ is an integral morphism. Thus, for each there exist $a_i\in I^i$, $i\in\{1,\hdots, n\}$ such that $$h(a_n)+\cdots+h(a_1)y^{n-1}+y^n=0.$$ 
	
	\noindent Hence, $a_n+\cdots+a_1x^{n-1}+x^n\in \ker h\sub \sqrt{(0)}$, which implies the existence of $r\in\bN$ such that $$(a_n+\cdots+a_1x^{n-1}+x^n)^r=0.$$
	
	Therefore, $x\in\overline{I}$.
	
	\end{proof}
	
	Finally, we present the main theorem of this work, where Lipman showed that relative Lipschitz saturation always gives rise to a radical algebra.

	\begin{teo}[\cite{lipman1}]
	Consider the sequence of ring morphisms $R \rightarrow A \overset{g}{\rightarrow} B$, and suppose that $g$ is an integral morphism. Then $A^{*}_{B,R}$ is a radicial $A$-algebra.
	\end{teo}
	
	\begin{proof}
	By Theorem \ref{202305262023} we have to check if the kernel of the canonical morphism $$\gamma:A^{*}_{B,R}\ten_AA^{*}_{B,R}\rightarrow A^{*}_{B,R}$$ \noindent is a nil ideal of $A^{*}_{B,R}\ten_AA^{*}_{B,R}$. Once $\ker\gamma$ is the ideal generated by $\{x\ten_A1_B-1_B\ten_Ax\mid x\in A^{*}_{B,R}\}$, it suffices to show that $\delta(x):=x\ten_A1_B-1_B\ten_Ax\in\sqrt{(0_{A^{*}_{B,R}\ten_AA^{*}_{B,R}})}, \forall x\in A^*_{B,R}$. Before to take care of it, observe that since $g$ is an integral morphism then $B$ is integral over $g(A)$, and we already know that $g(A)\sub A^{*}_{B,R}$, hence $B$ is integral over $A^{*}_{B,R}$. Consequently, the inclusion $\iota: A^{*}_{B,R}\hookrightarrow B$ is an integral morphism of $A$-algebras. By Lemma \ref{202305241953}\footnote{Notice that to apply this lemma as we have done, we do not have to be concerned whether $\ker\iota$ is a nil ideal or not.} we can conclude that the kernel of the map $\iota\ten_A\iota:A^{*}_{B,R}\ten_AA^{*}_{B,R}\rightarrow B\ten_AB$ is a nil ideal of $A^{*}_{B,R}\ten_AA^{*}_{B,R}$.
	
	Finally, let us check what remains. If $x\in A^{*}_{B,R}$ then $\Delta(x)\in\overline{\ker\varphi}$, and there exist $a_i\in(\ker\varphi)^i, i\in\{1,\hdots,n\}$ such that $$(\Delta(x))^n+a_1(\Delta(x))^{n-1}+\cdots+a_{n-1}\Delta(x)+a_n=0_{B\ten_AB}.$$ 
	
	Applying $\varphi$ in the last equation we obtain $(\varphi(\Delta(x)))^n=0_{B\ten_AB}$. It is easy to see that $(\iota\ten_A\iota)(\delta(x))=\varphi(\Delta(x))$, hence $(\iota\ten_A\iota)(\delta(x))\in\sqrt{(0_{B\ten_AB})}$. Since $\ker(\iota\ten_A\iota)$ is a nil ideal then by Lemma \ref{202305261918} we conclude that $\delta(x)\in \sqrt{(0_{A^{*}_{B,R}\ten_AA^{*}_{B,R}})}$, which ends the proof.	
	\end{proof}

\bibliographystyle{plain}
\bibliography{referencias}

\end{document}